\documentclass[microtype]{zgtpart}

\usepackage{hyperref}
\hypersetup{
colorlinks = true,
linkcolor = {blue},
urlcolor = {red},
citecolor = {blue}
}

\usepackage{tocloft}

\usepackage{setspace}
\usepackage[dvipsnames]{xcolor}
\usepackage[english=american]{csquotes}
\usepackage{enumerate}
\usepackage{textcomp}
\usepackage[normalem]{ulem}
\usepackage{mathrsfs}
\usepackage{amssymb}

\usepackage{mathtools}
\usepackage{stmaryrd}
\usepackage{caption}
\usepackage{amsmath} %for \pmb
\usepackage{amssymb}
\usepackage{tikz}

\usepackage{siunitx}
\usepackage{xcolor}
\usepackage{booktabs,colortbl, array}
\usepackage{pgfplotstable}

    {}%numberless section
\newcommand{\Irm}{\mathrm{I}}

\newcommand{\Trm}{\mathrm{T}}

\newcommand{\Acal}{\mathcal{A}}
\newcommand{\Bcal}{\mathcal B}

\newcommand{\Fcal}{\mathcal{F}}

\newcommand{\Hcal}{\mathcal{H}}

\newcommand{\Mcal}{\mathcal{M}}

\newcommand{\Ffrak}{\mathfrak{F}}

\newcommand{\Sfrak}{\mathfrak{S}}

\newcommand{\flatS}{\mathfrak{F}}

\newcommand{\Cscr}{\mathscr{C}}

\newcommand{\Pscr}{\mathscr{P}}

\newcommand{\Wscr}{\mathscr{W}}

\newcommand{\Abf}{\mathbf{A}}
\newcommand{\Bbf}{\mathbf{B}}
\newcommand{\Cbf}{\mathbf{C}}

\newcommand{\Ebf}{\mathbf{E}}

\newcommand{\Gbf}{\mathbf{G}}

\newcommand{\Sbf}{\mathbf{S}}

\newcommand{\Ebb}{\mathbb{E}}

\DeclareMathOperator{\Reg}{Reg}

\DeclareMathOperator{\dist}{dist}

\newcommand{\set}[2]{\left\{\, #1 \  \textup{\textbf{:}}\  #2 \,\right\}}

\newcommand{\N}{\mathbb{N}}
\newcommand{\loc}{\mathrm{loc}}

\newcommand{\spt}{\mathrm{spt}}

\newcommand{\Sing}{\mathrm{Sing}}

\newcommand{\toweakstar}{\overset{*}\rightharpoonup}

\newcommand{\mres}{\mathbin{\vrule height 1.6ex depth 0pt width
0.13ex\vrule height 0.13ex depth 0pt width 1.3ex}}

\newcommand{\eps}{\epsilon}

\newcommand{\proofstep}[1]{\textit{#1}}

\renewcommand{\eps}{\varepsilon}

\newcommand{\mbf}{{\textit{\text{\textbf{m}}}}}

\title{Hausdorff measure bounds for density-$Q$ flat singularities of minimizing integral currents}

 \author{Gianmarco Caldini and Anna Skorobogatova}
\givenname{}
\surname{}
 \address{}
 \email{}

\newtheorem{theorem}{Theorem}[section]
\newtheorem{lemma}[theorem]{Lemma}
\newtheorem{proposition}[theorem]{Proposition}

\theoremstyle{definition}

\newtheorem{remark}[theorem]{Remark}

\newtheorem{assumption}[theorem]{Assumption} 

\usepackage{stmaryrd}

\begin{document}

\begin{abstract}    % type your abstract below
In this article we prove that the set of flat singular points of locally highest density of area-minimizing integral currents of dimension $m$ and general codimension in a smooth Riemannian manifold $\Sigma$ has locally finite $(m-2)$-dimensional Hausdorff measure. In fact, the set of such flat singular points can be split into a union of two sets, one of which we show is locally $\mathcal{H}^{m-2}$-negligible, while for the other we obtain local $(m-2)$-dimensional Minkowski content bounds.
\end{abstract}

\maketitle

%%%%%%%%%%%%%%%%%%%%   Start of main body of article

\setlength{\cftbeforesecskip}{4pt}
\tableofcontents

\section{Introduction}
The problem of determining the size and structure of the singular set of area-minimizing integral currents of dimension $m$ and general codimension in a smooth Riemannian manifold $\Sigma$ is one of the most relevant questions in the regularity theory of generalized surfaces which are solutions to the Plateau problem. Thanks to the fundamental works of Almgren and De Lellis \& Spadaro, \textit{cfr.} \cite{Almgren83, Almgren2000, dls1, dls2, dls3}, followed by, more recently, the ones of Krummel \& Wickramasekera and the second author with De Lellis \& Minter, see \cite{KW1, KW2, KW3, dlsk1, dlsk2, dlmsk}, it is now known that the interior singular set of such area-minimizing integral currents is $(m-2)$-rectifiable (we refer to Section \ref{ss:literature} for a more detailed discussion of the existing literature of this problem). However, this leaves open the question of whether one can obtain local quantitative bounds on the size of the singular set or, more modestly, of some parts of it; in this article we make progress in this direction, showing that a specific meaningful subset of the whole singular set, that is the set of \emph{flat singular points of locally highest density}, has locally finite $(m-2)$-dimensional Hausdorff measure.

Since our statements are local and invariant under dilations and translations, our underlying assumption throughout this paper is the following. We refer the reader to \cite{Federerbook, Simon_GMT} for the theory of integral currents. Most of our notation is adopted from the articles \cite{dls1,dls2,dls3}; we refer the reader therein for any notation or terminology that is not defined here.

\setcounter{theorem}{-1}

\begin{assumption}\label{a:main} Let $\kappa_0\in (0,1]$, $m,l \in \mathbb{N}_{\geq 1}$ and $n\geq \overline n \geq 2$ be integers. Let $\Sigma$ be an $(m + \overline{n})$-dimensional embedded complete submanifold of $\mathbb R^{m+n} = \mathbb R^{m+\overline n + l}$ of class\footnote{In light of the recent work \cite{NardulliResende}, it is possible that the requirement of $C^{3,\kappa_0}$ regularity of $\Sigma$ might be weakened to merely $C^{2,\kappa_0}$, but we do not pursue this here.} $C^{3,\kappa_0}$, and let $T$ be an $m$-dimensional integral current in $\Sigma\cap \Bbf_{7\sqrt{m}}$ with $\partial T\mres \Bbf_{7\sqrt{m}} = 0$. We assume $T$ is area-minimizing in $\Sigma\cap \Bbf_{7\sqrt{m}}$.\footnote{That is: $\spt(T) \subset \Sigma\cap \Bbf_{7\sqrt{m}}$ and $\mathbb{M}(T + \partial S) \geq \mathbb{M}(T)$ for every $(m+1)$-dimensional integral current $S$ with $\spt(S) \subset \Sigma\cap \Bbf_{7\sqrt{m}}$.} Moreover, for every $p \in \Sigma\cap\Bbf_{7\sqrt{m}}$ we assume that $\Sigma \cap \Bbf_{7\sqrt{m}}$ is the graph of a $C^{3,\kappa_0}$ function $\Psi_p : \Trm_p\Sigma \cap \Bbf_{7\sqrt{m}} \to \Trm_p\Sigma^\perp$. We denote $${\textit{\text{\textbf{c}}}}(\Sigma) :=\sup_{p \in \Sigma \cap \Bbf_{7\sqrt{m}}}\|D\Psi_p\|_{C^{2,\kappa_0}}$$ and we assume ${\textit{\text{\textbf{c}}}}(\Sigma)\leq \overline{\eps}\leq 1,$
    where $\overline\eps$ is a small positive constant whose choice will be specified in each statement.
\end{assumption}

For $T$ and $\Sigma$ as in Assumption \ref{a:main}, we recall that a point $p\in \spt(T)$ is called a \emph{regular point} if there is a positive radius $r>0$, a $C^{3,\kappa_0}$-regular embedded $m$-dimensional oriented submanifold $\mathcal{M} \subset \Sigma$ and a positive integer $Q$ such that $T \mres \Bbf_r(p)=Q \llbracket \mathcal{M} \rrbracket$. The set of regular points of $T$, which is relatively open in $\operatorname{supp}(T)$, is denoted by $\Reg(T)$. Its complement, \emph{i.e.} $\spt(T) \setminus \Reg(T)$, is denoted by $\Sing(T)$ and is called the \emph{singular set} of $T$. For $Q \in \mathbb{N}$, we denote by $\mathrm{D}_Q(T)$ the points of density $Q$ of the current $T$, and set $\text{Sing}_Q(T):=\text{Sing}(T) \cap \mathrm{D}_Q(T).$

For any $r>0$ and $p \in \mathbb{R}^{m+n}, \iota_{p, r}$ : $\mathbb{R}^{m+n} \rightarrow \mathbb{R}^{m+n}$ is the map $y \mapsto \frac{y-p}{r}$ and we denote $T_{p, r}:=\left(\iota_{p, r}\right)_{\sharp} T$, \emph{i.e.} the pushforward of $T$ by the map $\iota_{p, r}$. We also denote by $\Sigma_{p,r}$ the rescaled ambient manifold $\iota_{p,r}(\Sigma)$. The classical monotonicity formula of mass ratios (see \cite[Theorem 17.6]{Simon_GMT} and \cite[Lemma A.1]{dls1}) implies that, for every $r_k \downarrow 0$ and $p \in \spt(T)$, there is a subsequence (not relabelled) for which $T_{p, r_k}$ converges to an integral cycle $S$ which is a cone (\emph{i.e.}, $S_{0, r}=S$ for all $r>0$ and $\partial S=0$) and which is area-minimizing in $\mathbb{R}^{m+n}$. Such a cone is referred to as a \emph{a tangent cone} to $T$ at $p$. 

Recall that a tangent cone supported in an $m$-dimensional plane is called \emph{flat}, and a point $p \in \Sing(T)$ with at least one flat tangent cone is called \emph{a flat singular point}. We denote by $\Ffrak(T)$ the set of flat singular points. We remark that from the constancy theorem, see \cite[Theorem 26.27]{Simon_GMT}, a flat tangent cone at a singular point $q$ must be an oriented $m$-dimensional plane with positive integer density $\Theta(T,p)$; we remark that by Allard's Regularity Theorem, see \cite{AAvarifolds}, $\Theta(T,q)>1$. Hence, we can write the following subdivision $$\Ffrak(T)= \bigcup_{Q=2}^\infty\Ffrak_Q(T),$$
where $\Ffrak_Q(T):=\{p \in\Ffrak(T) : \Theta(T,p)= Q\}$ is the set of \emph{flat singular points of density} $Q$. The typical examples of flat singular points are branching singularities of area-minimizers induced by complex subvarieties of $\C^n$; note moreover that the uniqueness of tangent cones is still unknown at flat singular points, even under the stronger assumption that all tangent cones at the considered point are flat. Following the notation adopted in \cite{dlsk1, dlsk2, dlmsk}, we introduce a further parameter which is a real number belonging to $[1,\infty)$, which is called the \emph{singularity degree of T at p} and that we denote $\Irm(T,p)$; see Section \ref{ss:sing-deg}. For a fixed $z \in [1,\infty)$, we will denote by $\Ffrak_{Q,\geqslant z}(T)$ the set of \emph{flat singular points of T with density Q and singularity degree} $\geq z$, that is $$\Ffrak_{Q,\geqslant z}(T):=\Ffrak_Q(T)\cap \{\Irm(T,p)\geq z\}.$$

We define the set $\Ffrak_{Q,\leqslant z}(T)$ analogously. By translating and rescaling, we may work with the following assumption throughout.

\begin{assumption}\label{a:mult-Q}
    Suppose that $T$, $\Sigma$ are as in Assumption \ref{a:main}. Moreover, suppose that $0$ is a flat singular point of $T$ and $Q\in \mathbb{N}\setminus\{0,1\}$ is the density of $T$ at $0$. Moreover, assume that there exists an $m$-dimensional plane $\pi_0\in T_0\Sigma$ such that $(\mathbf{p}_{\pi_0})_\sharp T \mres \Bbf_{6\sqrt{m}} = Q\llbracket B_{6\sqrt{m}}(0,\pi_0) \rrbracket$, where $\mathbf{p}_{\pi_0}$ is the orthogonal projection onto the plane $\pi_0$.
\end{assumption}

The aim of this article is to investigate the fine structural properties of $\Ffrak_Q(T)$ and, more formally, to prove the following theorem, which provides a Hausdorff measure bound on the flat singular points that have the highest density locally.

\begin{theorem}\label{t:content}
    Suppose that $T$ and $\Sigma$ satisfy Assumption \ref{a:mult-Q}. Then $\Ffrak_Q(T)$ has finite $(m-2)$-dimensional Hausdorff measure, namely
    \begin{equation}\label{e:content}
        \mathcal{H}^{m-2}(\Ffrak_{Q}(T))<\infty.
    \end{equation}
\end{theorem}

Note that Theorem \ref{t:content} yields merely a local Hausdorff measure bound, since we have already localized $T$ in Assumption \ref{a:mult-Q}.

\begin{remark}
    An analogous argument as the one developed in this article can be easily adapted to the case of integral currents which are semicalibrated by a smooth differential form, strengthening the rectifiability result obtained in \cite{Semicalibrated} to local Hausdorff measure bounds for density-$Q$ flat singularities as well.
\end{remark}

In fact, we demonstrate following two stronger results for the two respective pieces that we subdivide our flat singularities of density $Q$ into, from which Theorem \ref{t:content} is a simple consequence (see Section \ref{s:singularitiesgeq}). The first is the following, which states that we in fact have $(m-2)$-dimensional Minkowski content bounds for all flat singularities of singularity degree strictly larger than 1:

\begin{theorem}\label{p:contentfromNV}
    Suppose that $T$ and $\Sigma$ satisfy Assumption \ref{a:mult-Q} and $\delta\in (0,1/Q)$. Then $\flatS_{Q, \geqslant 1+\delta}(T)$ has finite $(m-2)$-dimensional (upper) Minkowski content, namely there exists $r_1=r_1(m,n,Q,\delta)$ such that and $C(Q,m,n,\delta)>0$ such that
    \begin{equation}\label{e:content}
        |\Bbf_r(\flatS_{Q, \geqslant 1+\delta}(T))| \leq C\, r^{n+2} \qquad \forall\, r\in (0,r_1]\,.
    \end{equation}
\end{theorem}

The second result states that the flat singularities of degree sufficiently close to 1 are $\Hcal^{m-2}$-negligible:

\begin{theorem}\label{p:nullset}
    Suppose that $T$ and $\Sigma$ satisfy Assumption \ref{a:mult-Q} and $\delta\in (0,1/Q)$. Then $$\mathcal{H}^{m-2}\left(\mathfrak{F}_{Q,\leqslant 1+\delta}\right)=0.$$
\end{theorem}

It may be also possible to prove that the entirety of $\Ffrak_Q(T)$ has locally finite $(m-2)$-dimensional Minkowski content, namely that \eqref{e:content} holds for $\flatS_Q(T)$ in place of $\flatS_{Q, \geqslant 1+\delta}(T)$. This would require one to show that for some $\delta=\delta(Q, m, n)<1/Q$ the set $\mathfrak{F}_{Q,\leqslant 1+\delta}(T)$ itself has locally finite Minkowski content, and not just $\mathcal{H}^{m-2}$-null measure. However, we believe that the covering argument of \cite{Simoncylindrical}, which is used to demonstrate that this set is $\Hcal^{m-2}$-negligible, is not well-adapted to strengthen this conclusion to such a content bound. The main obstruction seems to be the regions where there are "holes" in the set (in the sense of \eqref{e:no-gaps} failing), which do not appear to be compatible with a Minkowski-type estimate. On the other hand, the covering argument used in \cite{Nabervaltorta} allows one to establish such a content bound due to the fact that the analogue of regions with holes therein is instead formulated in the sense of closeness to a subspace of dimension strictly less than $m-2$, therefore being better suited for $(m-2)$-dimensional content bounds, despite not providing uniqueness of tangents to the surface almost-everywhere. Indeed, we are able to establish content bounds in Theorem \ref{p:contentfromNV} precisely because this subset of $\flatS_Q(T)$ is the one for which we are able to exploit the methods of \cite{Nabervaltorta}. We do not know if the latter methods can be used for the set $\flatS_{Q, \leqslant 1+\delta}(T)$, since there is no single appropriate monotone quantity at all scales around these points.

\begin{remark}\label{r:content}
    Note that we are only able to obtain local Hausdorff measure bounds on the flat singular points that are locally of top density. The main obstructions to obtaining such bounds on the whole singular set arise from possible accumulations of lower density singularities to higher density ones, and accumulations of singularities of lower strata to higher strata (see e.g. \cite{whitestratification} for Almgren's stratification). To the knowledge of the authors, there are no explicit examples of such phenomena, but ruling out this possibility is a very delicate and difficult problem, which remains widely open.
\end{remark}

\subsection{Previous literature}\label{ss:literature}

The regularity theory for area-minimizing integral current in codimension 1 was achieved by the joint effort of several deep contributions from the late sixties, see \cite{Degiorgifrontiere, Degiorgibernstein, Flemingcod1, Almgrencod1, Simonscod1, BDGcodim1, Federercod1}, until the work of Simon in which he proved the $(m-7)$-rectifiability of the (interior) singular set, see \cite{Simon_rect}. It took then more than twenty years to settle the question about whether or not the singular set had locally finite $(m-7)$-dimensional Hausdorff measure: introducing a set of new deep ideas, in \cite{Nabervaltorta} Naber \& Valtorta proved, as a corollary of their theory, that the singular set area-minimizing integral current in codimension 1 has in fact locally finite Minkowski content.

The study of area-minimizing integral currents in codimension at least 2 differs drastically, mainly due to the present of \emph{flat singularities}. In his monumental work about the regularity of area-minimizing integer rectifiable currents in general codimension, Almgren proved that they are in fact supported on smooth submanifolds apart from a relatively closed (interior) singular set of Hausdorff codimension at least 2, see \cite{Almgren83, Almgren2000}. The profound ideas and the techniques developed by Almgren have been fully understood only recently, thanks to the work of De Lellis and Spadaro who simplified, clarified and improved Almgren's regularity theory, see \cite{dlsmams, dls1, dlssns, dls2, dls3}. A strengthening of the $(m-2)$-dimensional Hausdorff dimension bound has been then proved by the second author, showing that the upper Minkowski dimension of the singular set is at most $m-2$, see \cite{skoro}. Only recently, De Lellis, Minter \& the second author and, independently, Krummel \& Wickramasekera have been able to prove that the singular set of a general codimension area-minimizing current is $(m-2)$-rectifiable, and that the tangent cone is a unique superposition of planes in $\mathcal{H}^{m-2}$-a.e. points in the support of the current, see \cite{dlsk1, dlsk2, dlmsk, KW1, KW2, KW3}. It is still unknown whether the singular set of an area-minimizing integral current always has locally finite $\mathcal{H}^{m-2}$ measure (\textit{cfr.} Remark \ref{r:content}); this is the case for two dimensional currents, as shown by \cite{Chang, DLSS2, DLSS1, DLSS3}. The best available result without any dimensional restriction comes as a corollary of the construction by Krummel and Wickramasekera in \cite{KW1, KW2, KW3}. In addition to the aforementioned results, the authors achieve the further conclusion that (their analogous set for) the set of flat singular points $\Ffrak_{Q, \geqslant 1+\delta}(T)$ can be decomposed, in a sufficiently small neighborhood $U$ of a point of density $Q$, into the union of finitely many sets $F_1 \cup \ldots \cup F_N$, each of which has locally finite $\mathcal{H}^{m-2}$ measure; we remark that this decomposition does not yield the local finiteness of the measure of the whole set of flat singular points in $U$, since the sets $F_i$ are not closed \emph{a priori}.

\subsection{Overview of the proof}
Our proof is divided into two main parts: in the first one we obtain a quantitative version of the arguments in \cite{dlmsk, dlsk1}, concluding that the set $\mathfrak{F}_{Q,\leqslant 1+\delta}(T)$ is $\mathcal{H}^{m-2}$-null (in place of simply $\Ffrak_{Q, 1}(T)$). To this aim, we first obtain a uniform version of \cite[Proposition 7.2] {dlsk1} where the radius threshold $r_0$ for the decay is independent of $T$ and the center $x$, \textit{cfr.} Proposition \ref{p:tilt-decay}. Then, to conclude the $\mathcal{H}^{m-2}$-negligibility of $\mathfrak{F}_{Q,\leqslant 1+\delta}(T)$ for $\delta>0$ arbitrarily close to $1/Q$, we develop suitable modifications of the arguments in the final part of \cite{dlmsk}, exploiting the idea that, at points in $\mathfrak{F}_{Q,\leqslant 1+\delta}(T)$, all the coarse blow-ups are homogeneous with degree $d \in[1,1+\delta]$ and thus, for a sufficiently small $\delta$, close to $1$-homogeneous Dir-minimizers; in particular, we improve \cite[Lemma 14.1]{dlmsk} by proving a quantitative version of it, \textit{cfr.} Lemma \ref{l:dimension-drop}. This will allow us to apply the \emph{conical excess decay theorem}, see Theorem \ref{t:conical-excess-decay} and \cite[Theorem 2.5]{dlmsk}, to rescaled and translated currents $T_{q,r}$ with $q\in \mathfrak{F}_{Q,\leqslant 1+\delta}(T)$ and $r>0$ sufficiently small, hence achieving that $\flatS_{Q,\leqslant 1+\delta} (T)$ is an $\mathcal{H}^{m-2}$-null set. We further remark that as a consequence of Lemma \ref{l:dimension-drop} we are in fact able to prove that at $\Hcal^{m-2}$-a.e. flat singular point of density $Q$, the singularity degree is at least $1+1/Q$.

In the second part we show that $\Ffrak_{Q, \geqslant 1+\delta}(T)$ has locally finite Minkowski content bounds, relying on a new construction with respect to the arguments in \cite{dls3} and \cite{dlsk2} to be able to tackle low frequency and high frequency points together. This is achieved by means of a suitable \emph{$(1+\delta)$-stopping and restarting procedure}: we improve the construction of the intervals of flattening, with the key difference with respect to \cite{dls3, dlsk2} that we do not decompose $\Ffrak_{Q, \geqslant 1+ 1/Q}(T)$ into countably many subsets; instead, we work with the entirety of this set as a single piece: this allows us to avoid concentrations between each piece. This argument, whose motivation naturally arises with the goal of obtaining Minkowski content bounds for the set $\Ffrak_{Q, \geqslant 1+ 1/Q}(T)$, has been already used, only for convenience, in \cite{Semicalibrated}; we remark that in order to show $m-2$-rectifiability and $\mathcal{H}^{m-2}$-uniqueness of tangent cones for semicalibrated integral currents there is no need of using this construction since the one already present in \cite{dlsk2} would have been sufficient, \textit{cfr.} \cite[page 4, lines 22-27]{dlsk2} and \cite[page 4, line 23]{Semicalibrated}. 

The first and the second parts together allow to conclude that the flat singular points with highest density have locally finite $\mathcal{H}^{m-2}$ measure, proving Theorem \ref{t:content}. 

%In the third part, we further sharpen the previous construction developing a new simplified adaptation of Simon's covering argument, see \cite{Simoncylindrical}, to conclude Minkowski content bounds also for the set $\flatS_{Q,\leqslant 1+\delta} (T)$. The key property we rely on herein is the control on the tilting between optimal $(m-2)$-dimensional subspaces around each point in $\flatS_{Q,\leqslant 1+\delta}(T)$ coming from Theorem \ref{t:conical-excess-decay}(d) until the first scale at which the property \eqref{e:no-gaps} fails.

The paper is organized as follows. In Section \ref{s:prelim} we recall the main tools of \cite{dls3, dlsk1} and set up the notation. In Section \ref{s:tiltexcessdecay} we prove the sharpened version of the tilt-excess decay; in Section \ref{s:singularitydegree1plusdelta} we prove that $\flatS_{Q,\leqslant 1+\delta} (T)$ is an $\mathcal{H}^{m-2}$-null set, a byproduct of which is the fact that we are able to treat only flat $Q$-points of singularity degree at least $1+\delta$ in the succeeding section. Finally, in Section \ref{s:singularitiesgeq} we obtain Minkowski content bounds for flat singular points with density $Q$ and singularity degree bigger than $1 + \delta$, allowing us to conclude the proof of Theorem \ref{t:content}.

\subsection*{Acknowledgments}
The authors would like to thank Camillo De Lellis for suggesting the problem and for making this collaboration possible. We are also grateful to the Institute for Advanced Study in Princeton and The Institute for Theoretical Studies in Z\"urich for warm hospitality. The research of G.C. has been supported by the Associazione Amici di Claudio Dematté; A.S. is grateful for the generous support of Dr. Max R\"ossler, the Walter Haefner Foundation and the ETH Z\"urich Foundation.

\section{Setup and preliminary results}\label{s:prelim}

Our article builds on techniques developed in \cite{dls3, dlsk1}: this section aims at recalling the main tools and setting up some of the key notation.

For suitable rescalings of the current around a given flat singular point, we exploit the center manifold construction, see \cite{dls2}, which provides a good "regularized" approximation of the average of the sheets of the current at some given scale, and in turn provides a suitable graphical approximation of the current parameterized over the center manifold. However, a center manifold and corresponding graphical approximation constructed at a certain scale relative to the given center may no longer be suitable at smaller scales. Thus, around $x\in \flatS_Q(T)$ we need to introduce a \emph{stopping condition} for the center manifolds and a countable collection of disjoint intervals of radii $(s_j(x),t_j(x)] \subset (0,t_0]$, for $j\in \N$ and $t_0$ sufficiently small, referred to as \emph{intervals of flattening}, such that for $\eps_3 > 0$ fixed as in \cite[Assumption 2]{dls3} we have
\[
    \Ebf(T, \Bbf_{6\sqrt{m}r}(x)) \leq \eps_3^2, \qquad \Ebf(T, \Bbf_{6\sqrt{m}r}(x)) \leq C {\textit{\text{\textbf{m}}}}^{(j)}_{0}\, r^{2-2\delta_2} \qquad \forall r\in \left(\tfrac{s_j}{t_j},3\right],
\]
where
\begin{equation}\label{e:m_0}
    {\textit{\text{\textbf{m}}}}^{(j)}_{0}:=\max\{\Ebf(T,\Bbf_{6\sqrt{m}t_j}(x)), \overline\eps^2 t_j^{2-2\delta_2}\}
\end{equation}
with $\delta_2$ is fixed as in \cite[Assumption 1.8]{dls2} so that all theorems and preposition therein can be applied. We recall that $\Ebf(T, \Bbf_\rho(x))$ is the \emph{tilt excess} of $T$ in the ball $\Bbf_{\rho}(x)$, defined by
\[
    \Ebf(T, \Bbf_{\rho}(x)) = \Ebf(T_{x,\rho}, \Bbf_1) := \frac{1}{2\omega_m \rho^m}\inf_{\text{$m$-planes $\pi$}} \int_{\Bbf_{\rho}(x)} |\vec{T} - \vec{\pi}|^2 \, d\|T\|\,,
\]
where $\omega_m$ is the $m$-dimensional Hausdorff measure of the unit ball $B_1$ in any $m$-dimensional plane. The tilt excess in a cylinder $\Cbf_\rho(x,\varpi)$ of radius $\rho$ oriented by a plane $\varpi$ and centered around $x$ is defined analogously, see for example \cite[Definition 1.1]{dls2}. Therefore, from now on we will work under the following assumption, allowing us to iteratively produce the above sequence of intervals.

\begin{assumption}\label{a:main-2}
    $T$ and $\Sigma$ are as in Assumption \ref{a:main}. The parameter $\overline \eps$ is chosen small enough to ensure that ${\textit{\text{\textbf{m}}}}_{0,0} \leq \eps_3^2$.
\end{assumption}

We refer to a sequence of positive numbers $r_k$ with $r_k \downarrow 0$ as a \emph{blow-up sequence of radii} around a flat singular point $x \in \flatS_Q(T)$ if $T_{x,r_k}$ converges to a flat tangent cone $Q\llbracket \pi \rrbracket$ for some $m$-dimensional plane $\pi\subset T_x\Sigma$. If $x=0$, we omit any reference to the center. Note that, having fixed a blow-up sequence $(r_k)$, for every $k$ sufficiently large there is a unique index $j(k)$ for an interval of flattening around $x$ such that $r_k \in ]s_{j(k)}, t_{j(k)}]$. By composition with a small rotation of coordinates, we may assume that the $m$-dimensional planes $\pi_k$ over which we parameterize the center manifolds $\Mcal_{j(k)}$ are identically equal to the same fixed plane $\pi_0 \equiv \R^m\times \{0\} \subset T_0\Sigma$. We use the shorthand notation
    \[
        {\textit{\text{\textbf{m}}}}_{x, \,k} := {\textit{\text{\textbf{m}}}}^{(j(k))}_{x} = \max\{\Ebf(T_{x,t_{j(k)}}, \Bbf_{6\sqrt{m}}), \overline\eps^2 t_{j(k)}^{2-2\delta_2}\}.
    \]

\subsection{Compactness procedure}\label{ss:compactness} Let $T$ satisfy Assumption \ref{a:main-2}, let $x\in \flatS_Q(T)$, let $r_k\in (s_{j(k)}(x), t_{j(k)}(x)]$ be a blow-up sequence around $x$ and let $\frac{\overline{s}_k}{t_{j(k)}} \in \big]\frac{3r_k}{2 t_{j(k)}}, \frac{3r_k}{t_{j(k)}}\big]$ be the scale at which the reverse Sobolev inequality \cite[Corollary 5.3]{dls3} holds for $r = \frac{r_k}{t_{j(k)}}$ around $x$. Then let $\overline{r}_k \coloneqq \frac{2\overline{s}_k}{3t_{j(k)}} \in \big]\frac{r_k}{t_{j(k)}}, \frac{2r_k}{t_{j(k)}}\big]$. We in turn define the corresponding rescalings of $T$, as well as those of the ambient manifold $\Sigma$ and the center manifolds:
\[
    \overline{T}_k \coloneqq  (\iota_{x,\overline{r}_k t_{j(k)}})_\sharp T\mres \Bbf_{\frac{6\sqrt{m}}{\overline{r}_k}}, \qquad \overline{\Sigma}_k \coloneqq \iota_{x,\overline{r}_k} (\Sigma_{0,t_{j(k)}}), \qquad \overline{\Mcal}_k \coloneqq \iota_{0,\overline{r}_k} (\Mcal_{x,t_{j(k)}})\, ,
\]
where $\Mcal_{x,t_{j(k)}}$ denotes the center manifold associated to $T_{x,t_{j(k)}}\mres \Bbf_{6\sqrt{m}}$ (see \cite{dls3}). We additionally denote by $N_{x,j(k)}$ the associated $\Mcal_{x,t_{j(k)}}$-normal approximation. We let $\pmb{\Phi}_k(x) := (x,\pmb{\varphi}_k(\overline r_k x))$ denote the maps parameterizing the graphs of the rescaled center manifolds $\overline{\Mcal}_k$, where $\pmb{\varphi}_k$ is the map parameterizing the center manifold $\Mcal_{x,t_{j(k)}}$ over $B_3(\pi_0)$.
%and let
%\[
%    \overline{\pmb{m}}_0^{(k)} \coloneqq \max\{\pmb{c}(\overline{\Sigma}_k)^2, \Ebf(\overline{T}_k, \Bbf_{6\sqrt{m}})\}.
%\]
Define
\[
    \overline{N}_k: \overline{\Mcal}_k \to \R^{m+n}, \qquad \overline{N}_k(p) \coloneqq \frac{1}{\overline{r}_k} N_{0,t_{j(k)}}(\overline{r}_k p),
\]
and let
\[
    u_k \coloneqq \frac{\overline{N}_k \circ \textbf{e}_k}{\mathbf{h}_k}, \qquad u_k:\pi_k \supset B_3 \to \Acal_Q(\R^{m+n}),
\]
where $\textbf{e}_k$ is the exponential map at $p_k \coloneqq \frac{\pmb{\Phi}_k(0)}{\overline{r}_k} \in \overline{\Mcal}_k$ defined on $B_3 \subset \pi_k \simeq T_{p_k} \overline{\Mcal}_k$ and $\mathbf{h}_k \coloneqq \|\overline{N}_k\|_{L^2(\Bcal_{3/2})}$. The reverse Sobolev inequality of \cite[Corollary 5.3]{dls3} gives a uniform control on the $W^{1,2}$ norm of $u_k$ on $B_{3/2} (0, \pi_k)$.

Then, following the proof of \cite[Theorem 6.2]{dls3}, there exists a subsequence (not relabelled), a limiting $m$-plane $\pi_0$ and a Dir-minimizing map $u \in W^{1,2}(B_{3/2}(0, \pi_0),\Acal_Q(\pi_0^\perp))$ with $\pmb{\eta}\circ u = 0$ and $\|u\|_{L^2(B_{3/2})} = 1$, such that (after we apply a suitable rotation to map $\pi_k$ onto $\pi$)
\begin{equation}\label{eq:compactness}
    u_k \longrightarrow u \quad \text{strongly in $W^{1,2}_\loc\cap L^2$}.
\end{equation}

\subsection{Almgren's frequency function and singularity degree}\label{ss:sing-deg}
Given a Lipschitz cut-off function $\phi: [0,\infty) \to [0,1]$ which vanishes identically for $t$ sufficiently large, equals $1$ for $t$ sufficiently small and is monotone, recall the following smoothed variant of Almgren's frequency function $I_u(x,r)$ for multi-valued Dir-minimizers, which is more convenient for our purposes than its classical counterpart.
\begin{align*}
D_{u} (x,r) &:= \int |Du (y)|^2 \phi \left(\frac{|y-x|}{r}\right)\, dy\, ,\\
H_{u} (x,r) &:= -\int \frac{|u(y)|^2}{|y-x|} \phi' \left(\frac{|y-x|}{r}\right)\, dy\, , \\
I_{u} (x,r) &:= \frac{r\, D_{u} (x,r)}{H_{u} (x,r)}\, .
\end{align*}
The same computations showing the monotonicity of Almgren's frequency function for Dir-minimizers apply to the latter smoothed variant\footnote{Note that Almgren's frequency function corresponds to the choice $\phi = {\mathbf{1}}_{[0,1]}$.}, \textit{cfr}. for instance
\cite[Section 3]{dls2}. Moreover, it can be readily checked that $I_u(x,\cdot)$ is constant in $r$ when the map is radially homogeneous about $x$, and this constant is the degree of homogeneity of the map. It follows then from the arguments in \cite[Section 3.5]{dlsmams} that the limit
\[
I_{u} (x,0) = \lim_{r\downarrow 0} I_u (x, r)
\]
is independent of the choice of $\phi$. We will henceforth fix the following convenient specific choice of $\phi$:
\begin{equation}\label{e:def_phi}
\phi (t) =
\left\{
\begin{array}{ll}
1 \qquad &\mbox{for $0\leq t \leq \frac{1}{2}$}\\
2-2t \quad &\mbox{for $\frac{1}{2}\leq t \leq 1$}\\
0 &\mbox{otherwise}\, .
\end{array}
\right.
\end{equation}
When $x=0$, we will omit the dependency on $x$ for $I$ and related quantities, and we will merely write $I_{u}(r)$.

Any map $u$ as defined by the above compactness procedure, recentered at $x\in \flatS (T)$, is called a {\em fine blow-up} limit along the sequence $r_k$ at $x$. The set
\[
    \Fcal(T,x) \coloneqq \set{ I_{u}(0)}{\text{$u$ is a fine blow-up at $x$ along some $r_k \downarrow 0$}},
\]
is the \emph{set of (singular) frequency values of $T$ at $x$}. The \emph{singularity degree} of $T$ at a flat singular point $x$ is defined as 
\[
    \Irm(T,x) \coloneqq \inf \Fcal(T,x) \, .
\]

\section{Uniform tilt excess decay}\label{s:tiltexcessdecay}

In this section we obtain a quantitative version of the compactness arguments in \cite{dlsk1}. In particular, we prove the following tilt excess decay result, which is a sharpened version of \cite[Proposition 7.2]{dlsk1}. 

\begin{proposition}[Uniform tilt excess decay]\label{p:tilt-decay}
        Let $T, \Sigma$ be as in Assumption \ref{a:main-2} and $x\in \Fcal_Q(T)$ a flat singular $Q$-point of $T$. For every $I_0>1$, there exist constants $C(m,n,Q)>0$, $\alpha(I_0,m,n,Q)\in (0,2-2\delta_2)$ and $r_0(I_0,m,n,Q)>0$ such that if $\Irm(T,x) \geq I_0$ then
         \begin{equation}\label{e:excess-decay}
            \Ebf(T,\Bbf_r(x)) \leq C\left(\frac{r}{r_0}\right)^\alpha \max \{\Ebf(T,\Bbf_{r_0}(x)), \overline\eps^2 r_0^{2-2\delta_2}\} \qquad \forall r\in (0,r_0).
        \end{equation}
\end{proposition}

\begin{remark}
We note that the radius $r_0$ in Proposition \ref{p:tilt-decay} is crucially only dependent on the lower singularity degree bound $I_0$ and on geometric quantities, rather than the current $T$ and the center $x$.
\end{remark}

Proposition \ref{p:tilt-decay} and Section \ref{s:singularitydegree1plusdelta} represent together the first part of the proof of Theorem \ref{t:content}, showing that $\mathcal{H}^{m-2}\left(\mathfrak{F}_{Q,\leqslant 1+\delta}(T)\right)=0$, \textit{cfr.} Proposition \ref{p:nullset}.

\subsection{Diagonal coarse blowups}\label{ss:diagonal-coarse}
Let $T_k$, $\Sigma_k$ be respective sequences of currents and ambient manifolds satisfying Assumption \ref{a:main-2} and let $\Abf_k := \Abf_{\Sigma_k}$. Let $x_k \in \flatS_{Q}(T_k)$ and suppose that $r_k\in (s_{j(k)}(x_k), t_{j(k)}(x_k)]$. Denote further 
\[
    \overline{T}_k := (T_k)_{x_k,r_k}, \qquad \overline\Sigma_k:=(\Sigma_k)_{x_k,r_k},
\]
and assume that $\overline{T}_k\mres \Bbf_{6\sqrt{m}} \toweakstar Q\llbracket \pi_0 \rrbracket$ with $\pi_0 = \R^{m} \times \{0\} \subset \R^{m+n}$. Let $\pi_k\subset T_{0} \overline\Sigma_k$ be such that
\[
    \Ebf(\overline T_k, \Bbf_{8M},\pi_k) = \Ebf(\overline T_k, \Bbf_{8M})\,,
\]
where $M > 0$ is large enough such that $\Bbf_L \subset \Cbf_{4M\bar{r}_k}$ for any $L \in {\Wscr^{j(k)}}$ with $L\cap \overline{B}_{\bar{r}_k}(0,\pi_0) \neq \emptyset$ (cf. \cite{dls2} for the definitions). In light of the height bound \cite[Theorem 1.5]{spolaor_15}, for $k$ sufficiently large we have
\[
    \Ebf(\overline T_k, \Cbf_{4M},\pi_k) \leq \Ebf(\overline T_k, \Bbf_{8M}) =: E_k \to 0.
\]
By applying a small rotation, we may assume that $\pi_k \equiv \pi_0$. Assuming $\eps_3^2 \leq \eps_1$, where $\eps_1$ is the threshold of \cite[Theorem 2.4]{dls1}, we may apply the latter result to produce a strong Lipschitz approximation $f_k: B_1(\pi_0) \to \Acal_Q(\pi_0^\perp)$ for $\overline T_k$ in $\Cbf_{4M}$. We in turn define the normalizations
\begin{equation}\label{e:normalizations}
        \overline f_k := \frac{f_k}{E_k^{1/2}},
\end{equation}
and we work under the assumption
\begin{equation}\label{e:A-infinitesimal}
    \Abf_k^2 = o(E_k).
\end{equation}
Following the compactness procedure in \cite{skoro}, where the arguments of \cite{dls3} are adapted for a varying sequence of currents with varying blowup centers, we deduce that there exists a Dir-minimizer $\overline f:B_1(\pi_0)\to \Acal_Q(\pi_0^\perp)$ such that
\[
    \overline f_k \to \overline f \quad \text{strongly in $L^2 \cap W^{1,2}_\loc(B_1(\pi_0))$}.
\]

We refer to such a $\overline f$ as a \emph{diagonal coarse blow-up along $r_k$}.

A key intermediate result for the proof of Theorem \ref{p:tilt-decay}, which will also be useful in the Section \ref{s:singularitydegree1plusdelta}, is the following compactness result, which is a generalization of \cite[Proposition 4.1]{dlsk1}.

\begin{proposition}\label{p:coarse=fine}
    Let $T_k$, $\Sigma_k$ satisfy Assumption \ref{a:main-2}. Let $x_k \in \flatS_{Q}(T_k)$ and suppose that $r_k\in (s_{j(k)}(x_k), t_{j(k)}(x_k)]$ are radii satisfying
    \begin{equation}\label{e:coarse-fine-lower-bound-scales}
        \liminf_{k\to\infty}\frac{s_{j(k)}(x_k)}{r_k} > 0.
    \end{equation}
    Then \eqref{e:A-infinitesimal} holds and, up to extracting a subsequence, we can consider a diagonal coarse blow-up $\overline{f}$ and, up to another subsequence, a fine blow-up $u$. Denoting by $v$ the average-free part of $\overline{f}$, then there is a real number $\lambda>0$  such that $v=\lambda u$.
\end{proposition}

\begin{proof}
    The proof of Proposition \ref{p:coarse=fine} follows verbatim that of \cite[Proposition 4.1]{dlsk1}, when combined with the observation that the conclusions of \cite[Lemma 4.5]{dlsk1} remain unchanged in the current setting with varying blowup centers (\textit{cfr.} also \cite{skoro}).
\end{proof}

% \begin{remark}
% Upper bound on frequency needed, coming from lower bound not needed, since we are not decomposing dyadically and changing once we approach $2-\delta$.
% \end{remark}

\begin{proof}[Proof of Proposition \ref{p:tilt-decay}] 
Fix $I_0 > 1$ and assume that $\Irm(T,x) \geq I_0$. The main step of the proof is to demonstrate the following decay property (\textit{cfr.} \cite[(Dec)]{dlsk1}):

\begin{itemize}
    \item[\textbf{(Dec)}] There are $\varepsilon=\varepsilon(I_0,m,n,Q) \in (0, \varepsilon_3]$, $\alpha=\alpha\left(I_0, m, n, Q\right)>0$, $\kappa=\kappa(I_0,m,n,Q) \in \mathbb{N}$ and $\tau=\tau(I_0,m,n,Q) >0$ such that, if

\begin{equation}\label{e:excess-small-Dec}
    \mathbf{E}\big(T, \mathbf{B}_{6 \sqrt{m}\, t_k}(x)\big)<\varepsilon^2
\end{equation}

and $t_k \leq \tau$, then:
\begin{itemize}
\item[(a)]The intervals of flattening $( s_k, t_k], (s_{k+1}, t_{k+1}], \ldots,( s_{k+\kappa}, t_{k+\kappa}]$ satisfy $s_{k+j-1}=$ $t_{k+j}$ for $j=1, \ldots, \kappa$;
\item[(b)]$\pmb{m}_{x, k+\kappa} \leqslant\left(\frac{s_{k+\kappa}}{t_k}\right)^\alpha \pmb{m}_{x, k}$.
\end{itemize}
\end{itemize}

Note that $\varepsilon$ and $\kappa$ do \emph{not} depend on the point $x$, nor on the current $T$, but only on the lower bound $I_0$ for the singularity degree. This is the key difference between the argument herein and that in \cite[Proposition 7.1]{dlsk1}.

We first prove property \textbf{(Dec)}, and then we will prove that this implies the tilt excess decay conclusion \eqref{e:excess-decay}, that is Proposition \ref{p:tilt-decay}. Property (\textbf{Dec}) follows by a contradiction argument taking a sequence of currents $T_k$ and varying centers $x_k$, getting the contradiction with diagonal coarse blow-up. 

First we choose $\alpha<\min 2\{I_0-1,1-\delta_2\}$. The choice of $\tau$ and $\varepsilon$ are subordinate to $\kappa$, which will be chosen later: hence we fix $\kappa$ without specifying its choice and treat it as a constant in order to choose $\tau$ and $\varepsilon$. 

\proofstep{Proof of \textbf{(Dec)} (a).} We remark that, to show point (a), $\kappa$ is fixed and given by point (b); hence the proof of part (a) is analogous to that of \cite[Proposition 7.1 \& 7.2]{dlsk1}, but we recall it here for completeness.

We start imposing that $\tau$ is small enough so that $$\overline{\varepsilon}^2 \tau^{2-2 \delta_2} \leq \varepsilon^2.$$
Then we recall that $$\mathbf{E}(T, \mathbf{B}_{6 \sqrt{m}\,s_k}(x)) \leq C\left(\frac{s_k}{t_k}\right)^{2-2 \delta_2} {\textit{\text{\textbf{m}}}}_{x, k} \leq C {\textit{\text{\textbf{m}}}}_{x, k}=C \max \left\{\overline{\varepsilon}^2 t_k^{2-2 \delta_2}, \varepsilon^2\right\} \leq C \varepsilon^2
$$

for every $k\in \N$ such that $t_k \leq \tau$, where $C$ is a geometric constant independent of $\varepsilon$. In particular, if we choose $\varepsilon$ sufficiently small, we conclude that $\mathbf{E}(T, \mathbf{B}_{6 \sqrt{m}\, s_k}(x)) \leq \varepsilon_3^2$, which in turn forces $t_{k+1}=s_k$. Observe also that ${\textit{\text{\textbf{m}}}}_{x, k+1} \leq C {\textit{\text{\textbf{m}}}}_{x, k}$, where the latter is the same constant of the previous estimate. In particular, as long as $t_{k+i+1}=s_{k+i}$ for $i \in\{0, \ldots, j\}$, we get $\mathbf{E}(T, \mathbf{B}_{6 \sqrt{m} s_j}(x)) \leq C^{\,j} {\textit{\text{\textbf{m}}}}_{x, k}$. Since this must be repeated $\kappa$ times, under the assumption that $C^{\kappa_0} \varepsilon^2 \leq \varepsilon_3^2$, we get by induction that $t_{k+j+1}=s_{k+j}$ and ${\textit{\text{\textbf{m}}}}_{x,k+j+1} \leq C {\textit{\text{\textbf{m}}}}_{x,k+j} \leq C^{\,j+1} {\textit{\text{\textbf{m}}}}_{x, k}$.\\

\proofstep{Proof of \textbf{(Dec)} (b).} Here we need to be careful and check the dependency of the exponent $\kappa$, that we want to be uniform with respect to the points $x$.

We observe that to prove (b) it suffices to show
\begin{equation}\label{e:decay2}
\mathbf{E}(T, \mathbf{B}_{6 \sqrt{m} \,s_{k+\kappa-1}}(x)) \leq\left(\frac{s_{k+\kappa-1}}{t_k}\right)^\alpha {\textit{\text{\textbf{m}}}}_{x, k}.
\end{equation}

Indeed, if ${\textit{\text{\textbf{m}}}}_{x, k}=\overline{\varepsilon}^2 t_k^{2-2 \delta_2}$, since we have that $2-2 \delta_2>\alpha$, we get

$$
\begin{aligned}
{\textit{\text{\textbf{m}}}}_{x, k+\kappa} & =\max \left\{\mathbf{E}(T, \mathbf{B}_{6 \sqrt{m} \,s_{k+\kappa-1}}(x)), \overline{\varepsilon}^2 s_{k+\kappa-1}^{2-2 \delta_2}\right\} \leqslant\left(\frac{s_{k+\kappa-1}}{t_k}\right)^\alpha \overline{\varepsilon}^2 t_k^{2-2 \delta_2} \\
& =\left(\frac{s_{k+\kappa-1}}{t_k}\right)^\alpha {\textit{\text{\textbf{m}}}}_{x, k}.
\end{aligned}
$$

But if ${\textit{\text{\textbf{m}}}}_{x, k}=\mathbf{E}(T, \mathbf{B}_{6 \sqrt{m} \,t_k}(x))$, then $\mathbf{E}(T, \mathbf{B}_{6 \sqrt{m}\, t_k}(x)) \geq \overline{\varepsilon}^2 t_k^{2-2 \delta_2}$ and hence again

$$
\overline{\varepsilon}^2 s_{k+\kappa-1}^{2-2 \delta_2} \leq\left(\frac{s_{k+\kappa-1}}{t_k}\right)^\alpha \mathbf{E}(T, \mathbf{B}_{6 \sqrt{m}\, t_k}(x)) \leq\left(\frac{s_{k+\kappa-1}}{t_k}\right)^\alpha {\textit{\text{\textbf{m}}}}_{x,k+\kappa}.
$$

Towards proving \eqref{e:decay2}, we first notice that we may assume without loss of generality that
\[
    \Ebf(T,\Bbf_{6\sqrt{m}s_{k+\kappa-1}}(x)) \leq C^\kappa \left(\frac{s_{k+\kappa-1}}{t_k}\right)^{2-2\delta_2}\mbf_{x,k}\,,
\]
for an appropriate choice of a constant $\rho_\ell(\kappa,\alpha) >0$ such that $\frac{s_{k+\kappa-1}}{t_k} \leq \rho_\ell$. Indeed, taking any fixed choice of
\[
    \rho_\ell \leq C^{-\frac{\kappa}{2-2\delta_2-\alpha}}\,,
\]
we may absorb $C^\kappa$ on the right-hand side, retaining the desired exponent for the decay estimate (b).

Furthermore, from \cite[Proposition 2.2]{dls3}, we have $\frac{s_{k+\kappa-1}}{t_k} \leq 2^{-5\kappa}$, so we may restrict ourselves to the case when
\begin{equation}\label{e:bounds}
    \rho_\ell \leq \frac{s_{k+\kappa-1}}{t_k} \leq \rho_u = 2^{-5\kappa}\,.
\end{equation}
We are now in a position to prove \eqref{e:decay2} - under the assumption \eqref{e:bounds} - by contradiction, since in this case the hypothesis \eqref{e:coarse-fine-lower-bound-scales} of Proposition \ref{p:coarse=fine} will be satisfied for any choice of sequence of blow-up scales between $s_{k+\kappa-1}$ and $t_k$, and thus we will be able to extract a diagonal coarse blow-up along our contradiction sequence. More precisely, up to extracting a subsequence for the index of the starting scales, we suppose there exists sequences $x_k\in \flatS_Q(T_k)$ with $\Irm(T_k,x_k)\geq I_0$ and $t_k(x_k) \downarrow 0$ such that
\begin{equation}
    \mbf_{x,k} \downarrow 0 \qquad \text{but} \qquad \Ebf(T,\Bbf_{6\sqrt{m}s_{k+\kappa-1}}(x_k)) > \left(\frac{s_{k+\kappa-1}}{t_k}\right)^\alpha \mbf_{x,k}\,,
\end{equation}
while \eqref{e:bounds} holds true along the sequence. We note that the scales $s_{k+i}, t_{k+i}$ for $i\in \{0,\dots,\kappa-1\}$ depend implicitly on $k$. We now proceed exactly as in \cite[Proof of Proposition 7.1 \& 7.2]{dlsk1}, applying Proposition \ref{p:coarse=fine} in place of \cite[Proposition 4.1]{dlsk1} along the scales $r_k= \tfrac{3}{2}\sqrt{m} t_k$. This allows us to extract a diagonal coarse blow-up $\overline{f}$ along the sequence $r_k$ with average-free part $v$ satisfying $I_v(0)\geq I_0$ and that induces the estimate
\begin{align*}
    \Ebf(T_k, \mathbf{B}_{6\sqrt{m} \sigma t_k}(x_k)) &\leq 8^m \sigma^{2\alpha} \mathbf{E} (T_k, \mathbf{B}_{6\sqrt{m} t_k}(x_k)) + C (\mathbf{E} (T_k, \mathbf{B}_{6\sqrt{m} t_k}(x_k)) + t_k^2 \mathbf{A}_k^2)^{1+\gamma} \\
    &\leq C \mbf_{x_k,k}^{1+\gamma} {\leq \frac{1}{2} \rho_\ell^\alpha \mbf_{x_k,k}}
\end{align*}
for any $\sigma \in [\rho_\ell,\rho_u]$, provided that $8^m\rho_u^{2\alpha} \leq \tfrac{1}{2}\rho_\ell^\alpha$ and $\mbf_{x_k,k}^\gamma \leq \rho_\ell^\alpha$. Importantly, observe that these choices of $\rho_\ell,\rho_u$ (and hence $\kappa$) and $k$ large enough yields dependency only on $I_0,m,n$ and $Q$ (implicitly via $\alpha$).\\

\proofstep{Proof of \eqref{e:excess-decay}.}
Observe that it suffices to check that, up to decreasing $\eps$ further if necessary (with the same dependencies), there exists $\sigma = \sigma(I_0,m,n,Q)>0$ such that \eqref{e:excess-small-Dec} holds for every $t_k \leq \sigma$. Indeed, since $\sigma$ is independent of $x$, this will allow us to iterate \textbf{(Dec)} combined with \cite[Proposition 2.2 (iv)]{dls3} as in the proof of \cite[Proposition 7.2]{dlsk1} in order to conclude the decay estimate \eqref{e:excess-decay} with $r_0 = \min\{\sigma, \tau\}$.

Towards proving the validity of \eqref{e:excess-small-Dec} for $t_k$ below a \emph{uniform-in-$x$} threshold $\sigma(I_0,m,n,Q)$ and $\eps(I_0,m,n,Q)>0$ sufficiently small (possibly smaller than the previous threshold), we argue by contradiction, again aiming to take a diagonal coarse blow-up. Unlike in the proof of \textbf{(Dec)} (b), the contradiction will arise from the fact that we will obtain a coarse blow-up of degree 1, which cannot happen in light of the fact that $I_0>1$. Our argument follows the line of reasoning of the proof of \cite[Proposition 8.1]{dlsk1}, but since we are in a simpler setting and our notation differs greatly, we repeat it here.

If, regardless of the choice of $\eps>0$, \eqref{e:excess-small-Dec} fails to occur at a uniformly small scale $\sigma$ independent of $x$, then we have a sequence $\eps_j \downarrow 0$ and $x_{j,k} \in \Fcal_Q(T_{j,k})$ with $\Irm(T_{j,k},x_{j,k})\geq I_0$ and $t_{j,k}(x_{j,k}) \downarrow 0$ (again up to extracting a subsequence for the index $k$, with $j$ fixed, for $t_{j,k}$) such that
\[
    \Ebf(T_{j,k},\Bbf_{6\sqrt{m} t_{j,k}}(x_{j,k})) = \eps_j^2\,.
\]
This implies that, for each fixed $j\in \N$, up to extracting a subsequence in $k$, $T_{j,k}$ converges to a tangent cone $C_j$ which is non-flat. Since $C_j$ is a cone, this in turn implies
\[
    \lim_{k\to\infty} \frac{\Ebf(T_{j,k},\Bbf_{6\sqrt{m} t_{j,k}}(x_{j,k}))}{\Ebf(T_{j,k},\Bbf_{6\sqrt{m} s_{j,k}}(x_{j,k}))} = 1\,,
\]
where $s_{j,k}=s_{j,k}(x_{j,k})$ are the stopping scales associated to the starting scales $t_{j,k}$. Thus, we may extract a subsequence of indices $k(j)$ such that
\[
    \lim_{j\to\infty} \frac{s_{j,k(j)}}{t_{j,k(j)}} \geq c > 0\,.
\]
Since \eqref{e:coarse-fine-lower-bound-scales} is satisfied at the scales $t_{j,k(j)}$, we are now in a position to apply Proposition \ref{p:coarse=fine} to obtain a coarse blow-up $\overline{f}$ along these scales, whose average-free part $v$ is comparable to the corresponding fine blow-up. In particular, $I_v(0) \geq I_0$.

We will now proceed to show that $\overline{f}$ is $1$-homogeneous (and thus so is $v$), therefore reaching a contradiction to the fact that $I_0 >1$. Let $E_{j,k(j)} := \Ebf(T_{j,k(j)},\Bbf_{8Mt_{j,k(j)}}(x_{j,k(j)}))$ and recall the normalizations in \eqref{e:normalizations} such that
\[
    \bar{f}_{j,k(j)} = \frac{f_{j,k(j)}}{E_{j,k(j)}^{1/2}} = (\eps_j +o(1))^{-1} f_{j,k(j)}\,,
\]
for the Lipschitz approximations $f_{j,k(j)}$ of $\overline{T}_{j,k(j)}:=(T_{j,k(j)})_{x_{j,k(j)},t_{j,k(j)}}$ as constructed in Section \ref{ss:diagonal-coarse}, where $o(1)$ is a quantity which converges to zero as $j\to\infty$. Let $K_j\subset B_1(\pi_0)$ be the domain of graphicality, as given by \cite[Theorem 2.4]{dls1}, where
\begin{equation}\label{e:graphicality}
    \Gbf_{f_{j,k(j)}} \mres (K_j \times \pi_0^\perp) \equiv \overline{T}_{j,k(j)}\mres (K_j \times \pi_0^\perp)\,,
\end{equation}
and $|B_1 \setminus K_j| \leq C\eps_j^{2(1+\gamma_1)}$, where $\gamma_1>0$ is as in \cite[Theorem 2.4]{dls1}. Consider now the anisotropic rescalings $\lambda^a_j(x,y) := (x,\eps_j^{-1} y)$ for $(x,y)\in \pi_0\times \pi_0^\perp$. Since $T_{j,k(j)}$ converges to $C_j$ with $\Ebf(C_j,\Bbf_{6\sqrt{m} t_{j,k(j)}}(x_{j,k(j)}))=\eps_j^2$ as $k\to\infty$ for each $j$ fixed, then we conclude that $(\lambda^a_j)_\sharp \overline T_{j,k(j)} \mres (K_j\times \pi_0^\perp)$ converge in Hausdorff distance to a non-flat cone. In light of \eqref{e:graphicality} and the fact that $\eps_j^{-2}|B_1 \setminus K_j| \rightarrow 0$ as $j\rightarrow \infty$, we deduce that $\overline f$ is indeed 1-homogeneous. We once again refer the reader to the proof of \cite[Proposition 8.1]{dlsk1} for further details.
\end{proof}

\section{Flat singularities of degree $< 1+1/Q$}\label{s:singularitydegree1plusdelta}

In this section we conclude, together with Section \ref{s:tiltexcessdecay}, the first part in the proof of Theorem \ref{t:content}, showing that for every $\delta<1/Q$ we have that $\mathcal{H}^{m-2}\left(\mathfrak{F}_{Q,\leqslant 1+\delta}(T)\right)=0$. 

In particular, we suitably modify the arguments in the final part of \cite{dlmsk}, exploiting the idea that at points in $\mathfrak{F}_{Q,\leqslant 1+\delta}(T)$ all the coarse blow-ups are homogeneous with degree $d \in[1,1+\delta]$ and thus, for any $\delta<1/Q$, close to $1$-homogeneous Dir-minimizers. Towards this aim, in Lemma \ref{l:dimension-drop} we improve the dichotomy of \cite[Lemma 14.1]{dlmsk} by proving a quantitative version of it: this will allow us to apply the \emph{conical excess decay theorem} of \cite{dlmsk}, which we recall in Theorem \ref{t:conical-excess-decay}, to rescaled and translated currents $T_{q,r}$ with $q\in \mathfrak{F}_{Q,\leqslant 1+\delta}(T)$ and $r>0$ sufficiently small, hence achieving that $\flatS_{Q,\leqslant 1+\delta} (T)$ is an $\mathcal{H}^{m-2}$-null set for any $\delta < 1/Q$, and thus $\mathcal{H}^{m-2}\left(\mathfrak{F}_{Q,< 1+1/Q}(T)\right)=0$ also.

Following the notation of \cite{dlmsk}, for $Q\in\N$, we introduce the notation $\Cscr(Q)$ for subsets of $\mathbb R^{m+n}$ consisting of unions of $N$ $m$-dimensional planes $\pi_1, \ldots, \pi_N$ with $N\leq Q$ and such that
\begin{itemize}
    \item $\pi_i \cap \pi_j$ being the same fixed $(m-2)$-dimensional plane $V$ for each $i<j$;
    \item $\pi_i \subset \varpi$ for some $(m+\overline n)$-dimensional plane $\varpi$.
\end{itemize}
If $p\in \Sigma$, then $\mathscr{C} (Q, p)$ will in turn denote the subset of $\mathscr{C} (Q)$ for which $\varpi = T_p \Sigma$.

We further let $\mathscr{P}$ and $\mathscr{P} (p)$ denote those elements of $\mathscr{C} (Q)$ and $\mathscr{C} (Q,p)$ respectively which consist of a single plane; namely, with $N=1$. For $\mathbf{S}\in \mathscr{C} (Q)\setminus \mathscr{P}$, the associated $(m-2)$-dimensional plane $V$ described above is referred to as the {\em spine of} $\mathbf{S}$ and will often be denoted by $V (\mathbf{S})$.

Given a ball $\Bbf_r(q) \subset \R^{m+n}$ and a cone $\mathbf{S}\in \mathscr{C} (Q)$, we define the \emph{one-sided conical $L^2$ height excess of $T$ relative to $\Sbf$ in $\Bbf_r(q)$}, denoted $\hat{\Ebf}(T, \mathbf{S}, \Bbf_r(q))$, by
	\[
		\hat{\Ebf}(T, \mathbf{S}, \Bbf_r(q)) \coloneqq \frac{1}{r^{m+2}} \int_{\Bbf_r (q)} \dist^2 (p, \mathbf{S})\, d\|T\|(p).
	\]
At the risk of abusing notation, we further define the corresponding \emph{reverse one-sided excess} as
 \[
\hat{\Ebf} (\mathbf{S}, T, \Bbf_r (q)) \coloneqq \frac{1}{r^{m+2}}\int_{\Bbf_r (q)\cap \mathbf{S}\setminus \Bbf_{ar} (V (\mathbf{S}))}
\dist^2 (x, {\rm spt}\, (T))\, d\mathcal{H}^m (x)\, ,
\]
where $a=a(m)$ is a dimensional constant, to be specified later. We subsequently define the \emph{two-sided conical $L^2$ height excess} as 
\[
    \mathbb{E} (T, \mathbf{S}, \Bbf_r (q)) :=
\hat{\Ebf} (T, \mathbf{S}, \Bbf_r (q)) + \hat{\Ebf} (\mathbf{S}, T, \Bbf_r (q))\, .
\]
We finally introduce the \emph{planar $L^2$ height excess} which is given by
\[
\Ebf^p (T, \Bbf_r (q)) := \min_{\pi\in \mathscr{P} (q)} \hat{\Ebf} (T, \pi, \Bbf_r (q))\, .
\]

We can state now the key \emph{conical excess decay theorem} from \cite[Theorem 2.5]{dlmsk}.

\begin{theorem}[Conical excess decay]\label{t:conical-excess-decay}
For every $Q,m,n$, $\overline n$, and $\varsigma>0$, there are positive constants $\varepsilon_0 = \varepsilon_0(Q,m,n,\overline n, \varsigma) \leq \frac{1}{2}$, $r_0 = r_0(Q,m,n,\overline n, \varsigma) \leq \frac{1}{2}$ and $C = C(Q,m,n,\overline n)$ with the following property. Assume that 
\begin{itemize}
\item[(i)] $T$ and $\Sigma$ are as in Assumption \ref{a:main};
\item[(ii)] $\|T\| (\Bbf_1) \leq (Q+\frac{1}{2}) \, \omega_m$;
\item[(iii)] There is $\mathbf{S}\in \mathscr{C} (Q, 0)\setminus\Pscr(0)$ such that 
\begin{equation}\label{e:smallness}
\mathbb{E} (T, \mathbf{S}, \Bbf_1) \leq \varepsilon_0^2 \,\mathbf{E}^p (T, \Bbf_1)\, 
\end{equation}
and 
\begin{equation}\label{e:no-gaps}
\Bbf_{ \varepsilon_0} (\xi) \cap \{p: \Theta (T,p)\geq Q\}\neq \emptyset, \, \text{ for every } \xi \in V (\mathbf{S})\cap \Bbf_{1/2}\, ;
\end{equation}
\item[(iv)] $\mathbf{A}^2 \leq \varepsilon_0^2 \, \mathbb{E} (T, \mathbf{S}', \Bbf_1),$ \, \text{for every} $\mathbf{S}'\in \mathscr{C} (Q, 0)$.
\end{itemize}
Then there is a $\mathbf{S}'\in \mathscr{C} (Q,0) \setminus \mathscr{P} (0)$ such that 
\begin{enumerate}%[\normalfont(a)]
    \item [\textnormal{(a)}] $\mathbb{E} (T, \mathbf{S}', \Bbf_{r_0}) \leq \varsigma \,  \mathbb{E} (T, \mathbf{S}, \Bbf_1)\,$; \\
    \item [\textnormal{(b)}] $\dfrac{\mathbb{E} (T, \mathbf{S}', \Bbf_{r_0})}{\mathbf{E}^p (T, \Bbf_{r_0})} 
\leq 2 \varsigma \,\dfrac{\mathbb{E} (T, \mathbf{S}, \Bbf_1)}{\mathbf{E}^p (T, \Bbf_1)}$; \\
    \item [\textnormal{(c)}] $\dist^2 (\Sbf^\prime \cap \Bbf_1,\Sbf\cap \Bbf_1) \leq C \, \mathbb{E} (T, \mathbf{S}, \Bbf_1)$;
    \item[\textnormal{(d)}] $\dist^2 (V (\mathbf{S}) \cap \Bbf_1, V (\mathbf{S}')\cap \Bbf_1) \leq C \, \dfrac{\mathbb{E}(T,\mathbf{S},\Bbf_1)}{\Ebf^p(T,\Bbf_1)}$\, .\label{e:spine-change}
    \end{enumerate}
\end{theorem}

Recall that in \cite[Theorem 2.5]{dlmsk}, Theorem \ref{t:conical-excess-decay} is applied to rescaled and translated currents $T_{q,r}$ with $q\in \flatS_{Q,1}(T)$ and $r>0$ sufficiently small. In order to do this, we need to verify that for most such points $q$ such rescalings satisfy \eqref{e:smallness} for some $\Sbf\in \Cscr(Q,0)\setminus \Pscr(0)$. This is indeed true, as demonstrated in \cite[Lemma 14.1]{dlmsk}: this is a consequence of \cite[Corollary 4.3]{dlsk1} and the classification of 1-homogeneous $Q$-valued Dir-minimizers on $\R^2$ as being superpositions of linear functions (see \cite[{Proposition 5.1}]{dlsmams}). Here we demonstrate that this in fact remains true for $q\in \flatS_{Q,< 1+1/Q}(T)$.

%\note{A: please check both the statement and proof when you have time, to make sure I didn't overlook something. I think we have to phrase it as follows, with the $\delta$ arbitrarily close to $1+1/Q$ rather than simply $p \in \flatS_{Q,< 1+\frac{1}{Q}} (T)$, because otherwise in statement (b) we do not ensure that all the points of singularity degree below $1+1/Q$ are in the neighborhood of the subspace (just the ones of degree $\leq \Irm(T,p)$). It is plausible that we can get around this with some kind of countable covering procedure in the end, but it would add yet another delicate element to the already delicate argument we have to try and adapt to aim for content 0 (or bounded) ... so doesn't seem worth it}
\begin{lemma}\label{l:dimension-drop}
%     Let $\delta \in (0,\tfrac{1}{Q})$ be fixed aribtrarily. 
    For each $\eps \in (0,1]$, the following holds. Suppose that $T$ and $\Sigma$ are as in Assumption \ref{a:main}. Then for each $p \in \flatS_{Q,< 1+1/Q} (T) \cap\Bbf_1$ there exists $\overline\rho=\overline\rho(p,\eps)>0$ such that the following dichotomy holds for each $\rho\in (0,\overline\rho]$: 
    \begin{itemize}
        \item[(a)] There exists $\Sbf\in \Cscr(Q,p)\setminus\Pscr(p)$ with
        \[
            (\rho\Abf)^2 + \Ebb(T,\Sbf,\Bbf_\rho(p)) \leq \eps^2 \Ebf^p(T, \Bbf_\rho(p));
        \]
        \item[(b)] There exists an $(m-3)$-dimensional affine subspace $V\subset T_p\Sigma$ (depending on $\rho$) such that
        \[
            \Sing_Q(T)\cap {\Bbf}_\rho(p) \subset \{q:\dist(q,V)< \eps\rho \}.
        \]
    \end{itemize}
\end{lemma}

\begin{remark}
Note that in the alternative (b), we may ask for the whole set of $Q$-points $\Sing_Q(T)$ to be contained in a small neighborhood of $V$, rather than just $\Ffrak_{Q, < 1+ 1/Q}(T)$, since it is simply a consequence of the fact that \emph{all} the density $Q$ points of $T$ are accumulating around either the spine of a cone that is $(m-3)$-invariant, or the spine of an $\Acal_Q$-valued Dir-minimizer that is $(m-3)$-invariant.
\end{remark}

\begin{proof}
    We proceed by contradiction via a compactness argument following a very similar line of reasoning to that in the proof of \cite[Lemma 14.1]{dlmsk}, only with the singularity degree lying in $[1,1+1/Q)$ instead of being equal to 1. The key observation will be that in the case when $T$ is close to planar and there is a coarse blow-up that is $(m-2)$-invariant, the classification of homogeneities for the blow-up in two dimensions (see \cite[Proposition 5.1]{dlsmams}) guarantees that the singularity degree must in fact be exactly equal to 1. Note that, in contrast to the compactness contradiction arguments in \cite[Lemma 14.1]{dlmsk}, we do not need to take a varying sequence of currents and centers. Indeed, since the scale $\bar\rho$ is allowed to (and in fact necessarily must be) dependent on the center $p$, we may simply fix $T$, $\Sigma$ and $p\in \flatS_{Q,< 1+1/Q}(T)\cap \Bbf_1$ when extracting the contradiction sequence; namely, for some $\eps \in (0,1]$, we assume that there exists a sequence of scales $\bar\rho_k\downarrow 0$ such that the rescaled currents $T_{p,\bar\rho_k}\mres \Bbf_{6\sqrt{m}}$ and their corresponding rescaled ambient manifolds $\Sigma_k := \Sigma_{p,\bar\rho_k}$ (recall the notation $\Abf_k := \Abf_{\Sigma_k}$) satisfy both
    \begin{equation}\label{e:m-2-inv-blowup}
        \Abf_k^2 + \Ebb(T_{p,\bar\rho_k},\Sbf,\Bbf_1) > \eps^2 \Ebf^p(T_{p,\bar\rho_k}, \Bbf_1) \qquad \forall \ \Sbf \in \Cscr(Q)\setminus\Pscr\,,
    \end{equation}
    and the property that for every $(m-3)$-dimensional subspace $V\subset T_0 \Sigma_{k}$ there exists a point $q_{k,V}$ such that%\note{A: if I am not wrong then actually it is just all the top density points in the small neighborhood of the spine in (b)}
    \begin{equation}\label{e:lower-spine-blowup}
        q_{k,V} \in \Sing_Q(T_{p,\bar\rho_k})\cap {\Bbf}_1 \setminus \{\dist(\cdot,V)< \eps\}\,.
    \end{equation}
    Up to extracting a further subsequence, we have two possible alternatives: either we have that $\Ebf^p(T, \Bbf_{6\sqrt{m}\bar\rho_k}(p)) \to 0$, or not. In the latter case, the sequence $T_{p, \bar\rho_k}\mres \Bbf_{6\sqrt{m}}$ converges weakly-$*$ (and in $L^2$-excess) to a non-flat tangent cone $C$. Let $V$ denote the subspace representing the spine of $C$. Since $C$ is not supported in an $m$-dimensional plane, $\dim V \leq m-2$. Since $\Abf_k^2 = \bar\rho_k^2 \Abf^2 \to 0$ but $\liminf_{k\to\infty}\Ebf^p(T_{p,\bar\rho_k}, \Bbf_{1}) \geq c_0 > 0$ (and noticing that this is comparable to the excess at the slightly larger scale $6\sqrt{m}$), the condition \eqref{e:m-2-inv-blowup} guarantees that we in fact must have $\dim V \leq m-3$. Taking this choice of $V$ in \eqref{e:lower-spine-blowup} yields a corresponding sequence of points $q_{k,V}$ with the stated property. Extracting yet another non-relabelled subsequence, we have that $q_{k,V} \to \bar q\in \Bbf_1\setminus \{\dist(\cdot,V)< \eps\}$. On the other hand, the upper semicontinuity of the density $\Theta(T_{p,\bar\rho_k},\cdot)$ along our sequence implies that $\Theta(C, \bar q) \geq Q$, and thus $\bar q$ must lie in the spine of $C$, yielding a contradiction.

    It remains to consider the alternative where $\Ebf^p(T, \Bbf_{6\sqrt{m}\bar\rho_k}(p)) \to 0$. In this case, due to the comparability of tilt and planar excesses, we may use \cite[Corollary 4.3]{dlsk1} (together with \cite[Theorem 2.10(vi)]{dlsk1} which a posteriori ensures that the hypothesis (21) therein holds) to extract a coarse blowup $f\in W^{1,2}(B_1(\pi_0);\Acal_Q(\pi_0^\perp \subset \R^{m+n}))$ of $T$ at $p$ that is a non-trivial, radially homogeneous Dir-minimizer of degree $\Irm(T, p) < 1+\tfrac{1}{Q}$ with $\boldsymbol\eta\circ f = 0$. In the process of taking this coarse blowup, we may assume by a rotation of coordinates that the Lipschitz approximations of $T_{p,\bar\rho_k}$ are always parameterized over the fixed plane $\pi_0$. Moreover, note that the validity of \cite[(3.4)]{dlsk1} guarantees that
    \begin{equation}\label{e:sff-scales-away}
        \frac{\bar\rho_k^2 \Abf_k^2}{\Ebf^p(T, \Bbf_{6\sqrt{m}\bar\rho_k}(p))} \to 0\,.
    \end{equation}
    Let $V \subset \pi_0$ denote the subspace of maximal dimension in which $f$ is translation-invariant. In light of \cite[Theorem 0.11]{dlsmams}, we again deduce that $\dim V \leq m-2$. If $\dim V = m-2$, letting $(x,y) \in V\times V^{\perp_{\pi_0}}\cong \R^{m-2}\times \R^2$ denote the corresponding splitting of the coordinates, $f(x,y)$ canonically identifies with $g(y)$ for an $\Irm(T,p)$-homogeneous Dir-minimizer $g\in W^{1,2}(B_1\subset \R^2;\Acal_Q(\R^n))$ (after additionally performing a rotation of the coordinates). Since $1\leq \Irm(T,p) < 1+\frac{1}{Q}$, the classification \cite[Proposition 5.1]{dlsmams} guarantees that $\Irm(T,p) = 1$ and $g$ (and therefore $f$) is a superposition of (at least two distinct) linear functions, inducing a cone $\Sbf\in \Cscr(Q)\setminus \Pscr$. Combining this with \eqref{e:m-2-inv-blowup}, \eqref{e:sff-scales-away}, the estimates in \cite[Theorem 2.4]{dls1} and the strong $L^2$-convergence of the Lipschitz approximations for $T_{p,\bar\rho_k}$ to $f$ after normalizing by $\Ebf(T, \Bbf_{6\sqrt{m}\bar\rho_k}(p))$, yields the desired contradiction in this case. On the other hand, if $\dim V \leq m-3$, we take this choice of $V$ in \eqref{e:lower-spine-blowup} to obtain the corresponding sequence of points $q_{k,V}$, which, by \cite[Theorem 2.7]{dls1} satisfy $\mathbf{p}_{\pi_0}(q_{k,V}) \to \bar z \in B_1(\pi_0)$ with $f(\bar z)=Q\llbracket \boldsymbol\eta\circ f(\bar z)\rrbracket = Q\llbracket 0\rrbracket$. This implies that $\bar z \in V$, contradicting the fact that $\dist(q_{k,V}, V) \geq \eps$ for each $k$.
\end{proof}

\begin{remark}
    We further remark that as a consequence of Lemma \ref{l:dimension-drop} we are in fact able to prove that at $\Hcal^{m-2}$-a.e. flat singular $Q$-point the singularity degree is at least $1+1/Q$. However, such a classification at \emph{all} such points remains open.
\end{remark}

As a corollary of Proposition \ref{p:tilt-decay} and Lemma \ref{l:dimension-drop} we immediately obtain Theorem \ref{p:nullset}.

\section{Flat singularities of degree $\geq 1 + 1/Q$}\label{s:singularitiesgeq}
Throughout this section, we fix $\delta\in (0,1/Q)$ arbitrarily. We follow the methods of \cite{dlsk2} to deduce the Minkowski content bounds for $\flatS_{Q, \geqslant 1+\delta}(T)$ claimed in Theorem \ref{p:contentfromNV}.

Recall that in \cite{dlsk2}, the first step is a decomposition of $\Ffrak_{Q, > 1}(T)$ into countably many pieces based on the value of $\Irm(T, \cdot)$. The key difference herein is that we do not need to decompose $\Ffrak_{Q, \geqslant 1+ \delta}(T)$ into countably many subsets, but we may rather just work with the entirety of this set as a single piece $\Sfrak = \Sfrak_{K_0}$ of the countable decomposition
\[
    \Ffrak_Q(T)\cap \Bbf_1 = \bigcup_{K \in \N} \Sfrak_K, \qquad \Sfrak_K := \{x\in \Ffrak_Q(T) : \Irm(T,x) \geq 1+2^{-K} \} \cap {\Bbf}_1\,.
\]
Indeed, we may simply choose $K_0(Q) \geq \frac{\log Q}{\log 2}$. In order to work with $\Sfrak_{K_0}$ in its entirety, without having to take a countable union of sets
\[
    \tilde\Sfrak_K := \{x\in \Ffrak_Q(T) : 2-\delta_2 - 2^{-K} \geq \Irm(T,x) \geq 1+2^{-K} \} \cap {\Bbf}_1, \qquad K \in \N\,,
\]
we rely on an adaptation of such a countable decomposition that was recently introduced by the second author with Minter, Parise \& Spolaor in \cite{Semicalibrated} as a way of simplifying the decomposition procedure. Therein such a refined decomposition procedure is not necessary and is merely used for convenience, whereas here it is crucial in order to obtain the content bound \eqref{e:content}. Indeed, the methods of \cite{dlsk2} yield such an estimate for each piece $\tilde\Sfrak_{K,J}$ of $\tilde\Sfrak_K$, where the index $J$ characterizes the first scale at which the tilt excess decay of \cite[Proposition 7.2]{dlsk1} kicks in (see \cite[Theorem 9.7]{dlsk2}), but do not rule out the possibility of Minkowski content concentration between these sets since they are not closed. Since the main improvement of our Proposition \ref{p:tilt-decay} relative to \cite[Proposition 7.2]{dlsk1} is precisely that the starting scale $r_0$ is dependent only on a lower bound $I_0$ on the singularity degree, we are able to avoid such a countable decomposition.

For the purpose of clarity, we provide the setup and describe the key ideas here and we refer the reader to \cite{dlsk2}, \cite[Part 2]{Semicalibrated} for the full proof.

First of all, note that Proposition \ref{p:tilt-decay} guarantees that for each $x\in \Sfrak$, there exists $r_0(\delta,m,n)>0$ such that for $\mu=2^{-K_0-1}$ and $0< r< s< r_0$ we have
\begin{equation}\label{e:tilt-decay-NV}
    \Ebf(T,\Bbf_r(x)) \leq \left(\frac{r}{s}\right)^{2\mu} \max\{\Ebf(T,\Bbf_{s}(x)), \bar\eps^2 s^{2-2\delta_2}\}\,.
\end{equation}

Now, consider a point $x\in \Sfrak$. Observe that \cite[Theorem 2.10]{dlsk1} tells us that $x$ has corresponding intervals of flattening $\{(t_{k+1},t_k]\}_{k\geq 0}$ with $\inf_k \frac{t_{k+1}}{t_k} > 0$. Following the setup in Section \ref{s:prelim} with the fixed center $x$, consider the corresponding center manifolds $\Mcal_{x,k}$ and normal approximations $N_{x,k}$, together with a geometric blow-up sequence of scales $\{\gamma^j\}_{j\geq 0}$. 

For $j=0$, let $\widetilde\Mcal_{x,0} =\Mcal_{x,0}$ and $\widetilde N_{x,0} = N_{x,0}$. For $j=1$, if $\gamma$ lies in the same interval of flattening to $t_0=1$, let 
\[
    \widetilde\Mcal_{x,1} := \iota_{0,\gamma}(\Mcal_{x,0}), \qquad \widetilde N_{x,1}(x) := \frac{N_{x,0}(\gamma x)}{\gamma}.
\]
Otherwise, let $\widetilde \Mcal_{x,1}$ be the center manifold associated to $T_{x,\gamma}\mres\Bbf_{6\sqrt{m}}$, with corresponding normal approximation $N_{x,1}$. 

For each $j\geq 2$, define $\widetilde\Mcal_{x,j}$ inductively as follows. If $\gamma^{j}$ lies in the same flattening as $\gamma^{j-1}$, let
\[
    \widetilde\Mcal_{x,j} := \iota_{0,\gamma}(\Mcal_{x,j-1}), \qquad \widetilde N_{x,j}(x) := \frac{N_{x,j-1}(\gamma x)}{\gamma}.
\]
Otherwise, let $\widetilde \Mcal_{x,j}$ be the center manifold associated to $T_{x,\gamma^j}\mres\Bbf_{6\sqrt{m}}$, with corresponding normal approximation $N_j$.

It follows from the excess decay \eqref{e:tilt-decay-NV} that around any $x\in \Sfrak$, we may replace the procedure in Section \ref{ss:compactness} with the intervals $(\gamma^{k+1},\gamma^k]$ in place of $(s_k,t_k]$, and with $\pmb{m}_{0,k}$ therein instead defined by  
\begin{equation}\label{e:rid-of-m0}
\pmb{m}_{x,k} = \mathbf{E} (T_{x, \gamma^k}, \Bbf_{6\sqrt{m}}) =
\mathbf{E} (T, \Bbf_{6\sqrt{m} \gamma^{k}} (x))\, .
\end{equation}
Observe that in particular, if $x\in\Sfrak$ originally has finitely many intervals of flattening with $(0,t_{j_0}]$ being the final interval, it will nevertheless have infinitely many adapted intervals of flattening, but for all $k$ sufficiently large, $\widetilde\Mcal_{x,k}$ and $\widetilde N_{x,k}$ are arising as rescalings of $\Mcal_{x,j_0}$ and $N_{x,j_0}$ respectively.

Abusing notation, let us henceforth simply write $\Mcal_{x,k}$ for the center manifold $\widetilde\Mcal_{x,k}$, with its corresponding normal approximation $N_{x,k}$. 

The key ideas of \cite{dlsk2} can be roughly summarized in the following steps:

\begin{itemize}
    \item[(1)] Exploit radial and spatial variations of the frequency function for $N_{x,k}$ (\cite[Lemma 10.7, Lemma 11.4]{dlsk2}) in order to obtain a quantitative control on the deviation of $N_{x,k}$ from being radially homogeneous in a given geodesic annulus on $\Mcal_{x,k}$, in terms of the pinching of the frequency between the two radial scales of the annulus; see \cite[Proposition 11.2]{dlsk2}. Note that these estimates do not require that $N_{x,k}$.

    \item[(2)] Obtain quantitative spine splitting and frequency pinching along $(m-2)$-dimensional subspaces via compactness procedures that rely on the unique continuation properties of the limiting Dir-minimizers; see \cite[Lemma 12.1, Lemma 12.2]{dlsk2}.
    \item[(3)] Obtain quantitative control on the $\beta_2$ coefficients associated to a suitable discrete approximation of the measure $\Hcal^{m-2}\mres \Sfrak$ (see \cite[Proposition 13.2]{dlsk2}, which in turn leads to an estimate on its square function.
    \item[(4)] Combine the above steps (1)-(3) with an iterative covering procedure (see \cite[Appendix A]{dlsk2}, that involves effectively covering all the points in $\Sfrak$ where the universal frequency remains pinched close to its local maximal value, so that the frequency drops by a discrete amount outside of this cover, allowing one to iterate a finite number of times.
    
\end{itemize}

We hence conclude the proof of our main result.

\begin{proof}[Proof of Theorem \ref{t:content}]
The proof is now an immediate corollary of Proposition \ref{p:nullset} and Theorem \ref{p:contentfromNV}, after choosing $\delta$ in Theorem \ref{p:contentfromNV} strictly smaller than $\delta$ in Proposition \ref{p:nullset} to ensure an overlap between the two sets.
\end{proof}

\newpage

%%%%%%%%%%%%%%%%%%%%   End of main body of article
%
%                             References
%
%   BiBTeX users uncomment the following line:
%
%\bibliographystyle{gtart}
%

%\section*{List of Symbols}
%Here we list some of the most frequent notations of the article:
%\vspace{0.5cm}
%
%\begin{tabular}{ll}
%    $\Bbf_{r}(p)$ & Open ball with radius $r$ and center $p$ in $\mathbb{R}^{m+n}$;\\
%    $\Bbf_{r}$ & Open ball with radius $r$ and center $0$ in $\mathbb{R}^{m+n}$;\\
%    $B_{r}(p,\pi)$ & Open disk $\Bbf_{r}(p)\cap (p+\pi)$;\\
%    
%\end{tabular}
%
%
%\newpage

\end{document}